\documentclass[12pt,a4paper,psamsfonts]{amsart}
\usepackage{amssymb,amscd,amsxtra,calc}
\usepackage{cmmib57}

\theoremstyle{plain}
    \newtheorem{thm}{Theorem}[section]

    \newtheorem{theorem}[thm]{Theorem}

\theoremstyle{definition}

\theoremstyle{remark}

\title[Backtracking Gradient Descent on Banach spaces]{Some convergent results for Backtracking Gradient Descent method on Banach spaces}

\author{Tuyen Trung Truong}
\address{Matematikk Institut, Universitetet i Oslo, Blindern, 0851 Oslo, Norway}
\email{tuyentt@math.uio.no}
\thanks{}
\date{\today}
\begin{document}
\begin{abstract}
In a recent paper, the author showed that a version of Backtracking gradient descent method (called Backtracking gradient descent - New) can effectively avoid saddle points, when applied to functions on Euclidean spaces whose gradients are locally Lipschitz continuous. For example, this result is valid for $C^2$ functions. 

In this paper, we extend the mentioned result to infinite dimensional Banach space. Our result concerns the following condition: 

{\bf Condition C.} Let $X$ be a Banach space. A $C^1$ function $f:X\rightarrow \mathbb{R}$ satisfies Condition C if whenever $\{x_n\}$ weakly converges to $x$ and $\lim _{n\rightarrow\infty}||\nabla f(x_n)||=0$, then $\nabla f(x)=0$.   

Condition C is satisfied when $X$ is finite dimensional, or when  $f$ is quadratic or convex. Another interesting case for Condition C in infinite dimensions is when $\nabla f$ is of class $(S)_+$ defined by Browder in Nonlinear PDE.

We assume that there is given a canonical isomorphism between $X$ and its dual $X^*$, for example when $X$ is a Hilbert space. A main result in the paper, which can be concisely formulated, is as follows: 

{\bf Theorem.} Let $X$ be a reflexive Banach space and $f:X\rightarrow \mathbb{R}$ be a $C^2$ function which satisfies Condition C. Moreover, we assume that for every bounded set $S\subset X$, then $\sup _{x\in S}||\nabla ^2f(x)||<\infty$. We choose a point $x_0\in X$ and construct by the Local Backtracking GD procedure (which depends on $3$ hyper-parameters $\alpha ,\beta ,\delta _0$, see later for details) the sequence $x_{n+1}=x_n-\delta (x_n)\nabla f(x_n)$. Then we have: 

1) Every cluster point of $\{x_n\}$, in the {\bf weak} topology, is a critical point of $f$. 

2) Either $\lim _{n\rightarrow\infty}f(x_n)=-\infty$ or $\lim _{n\rightarrow\infty}||x_{n+1}-x_n||=0$. 

3) Here we work with the weak topology. Let $\mathcal{C}$ be the set of critical points of $f$. Assume that $\mathcal{C}$ has a bounded component $A$. Let $\mathcal{B}$ be the set of cluster points of $\{x_n\}$. If $\mathcal{B}\cap A\not= \emptyset$, then $\mathcal{B}\subset A$ and $\mathcal{B}$ is connected.     

4) Assume that $f$ has at most countably many saddle points. Then for generic choices of $\alpha ,\beta ,\delta _0$ and the initial point $x_0$, if the sequence $\{x_n\}$ converges - in the {\bf weak} topology, then the limit point cannot be a saddle point. 

5) Assume that $X$ is separable. Then for generic choices of $\alpha ,\beta ,\delta _0$ and the initial point $x_0$, if the sequence $\{x_n\}$ converges - in the {\bf weak} topology, then the limit point cannot be a saddle point. 

Parts 1 to 3 are valid for all Banach space and other versions of Backtracking GD. The exceptional sets for $x_0$ in parts 4) and 5) are "shy".  We also discuss how to treat 4) and 5) using normalized duality mappings. 
\end{abstract}
\maketitle

Finding extrema of a function is an important task. Gradient Descent methods (GD) are a popular method used for this purpose, and indeed are best theoretically guaranteed among all iterative methods so far. There has been a lot knowledge accumulated - since 1847 - on this problem in the finite dimensional setting, in particular on Euclidean spaces.  For example, recently  in \cite{truong-nguyen, mbtoptimizer}, the author and collaborator proved convergence of Backtracking GD under very general conditions (valid for all Morse functions) and show that some of its modifications (Two-way Backtracking GD and Hybrid Backtracking GD) can be implemented effectively in Deep Neural Networks (DNN), with strong theoretical justifications, non-expensive, while being automatic and having very good performance compared to state-of-the-art algorithms such as Momentum, NAG, Adam, Adagrad, Adadelta and RMSProp. For maps $f$ whose gradient $\nabla f$ is Lipschitz continuous with Lipschitz constant $L$ (this class of functions is usually denoted by $C^{1,1}_L$), it has been shown \cite{lee-simchowitz-jordan-recht, panageas-piliouras} that Standard GD, with learning rate in the order of $1/L$, avoids generalised saddle points. Here we recall that a critical point $z$ is saddle if $f$ is $C^2$ near $z$ and the Hessian $\nabla ^2f(z)$ has both positive and negative eigenvalues. More generally, $z$ is a generalised saddle point if $\nabla ^2f(z)$ has at least one negative eigenvalue. This result on avoidance of generalised saddle points has  been extended in \cite{truong} to more general functions for the following version of Backtracking GD (called Backtracking GD-New in that paper, but here we change to another name which seems more suitable)

{\bf Local Backtracking GD:} Let $f:\mathbb{R}^k\rightarrow \mathbb{R}$ be a $C^1$ function. Assume that there are continuous functions $r,L:~\mathbb{R}^k\rightarrow (0,\infty)$ such that for all $x\in \mathbb{R}^k$, the gradient $\nabla f$ is Lipschitz continuous on $B(x,r(x))$ with Lipschitz constant $L(x)$. Fix constants $0<\alpha ,\beta <1$ and $\delta _0>0$. For each $x\in \mathbb{R}^k$, we define the number $\hat{\delta}(x)$ to be the largest number $\delta$ among $\{\beta ^n\delta _0:~n=0,1,2,\ldots \}$ which satisfy the following 2 conditions
\begin{eqnarray*}
\delta &<&\alpha /L(x),\\
\delta ||\nabla f(x)||&<&r(x).
\end{eqnarray*}

Note that when $x$ is very close to a critical point, then the second condition above is not needed. We can use Local Backtracking GD for example if $f\in C^{1,1}_L$ or if $f\in C^2$. In fact, if $f\in C^2$, then we can use for example $r(x)=1+||x||$ and $L(x)=\sup _{z\in B(x,r(x))}||\nabla ^2f(z)||$.  

Backtracking GD also allows learning rates $\delta _n$ to be unbounded from above \cite{truong2}. We can combine the many different versions of Backtracking GD, and still can prove rigorous theoretical results for such combinations. The literature on Gradient Descent methods in finite dimensional spaces is very abundant, we refer the readers to the above papers and references therein for more detail.

In this paper for simplicity we use only the Local Backtracking GD. Here we aim to extend the above mentioned results to infinite dimensional Banach spaces. For a comprehensive reference on Banach spaces, we refer the readers to \cite{rudin}. We recall that a Banach space $X$ is a vector space, together with a norm $||.||$, and is complete with respect to the norm. The topology induced by $||.||$ is called the strong topology. The dual space $X^*$, consisting of bounded linear maps $h:X\rightarrow \mathbb{R}$, is also a Banach space, with norm $||h||=\sup _{x\in X:~||x||\leq 1}|h(x)|$.   

There are many pathologies when one works with infinite dimensional Banach spaces. For example, a bounded sequence $\{x_n\}$ may not have any subsequence converging in the (strong) topology, unlike the case of finite dimensions where we have Bolzano-Weierstrass property. For to better deal with this, we need to work with another topology, called weak topology. To distinguish with the usual (strong) topology, we will use a prefix "w-" when speaking about weak topology. Here, a sequence $\{x_n\}$ w-converges to $x$ if for every $h\in X^*$ we have $\lim _{n\rightarrow\infty}h(x_n)=h(x)$. Note that strong and weak topologies coincide only if $X$ is finite dimensional. If we want to preserve the good Bolzano-Weierstrass property, then we need to require that $X$ is reflexive, in that $X$ is isomorphic to its double dual $X^{**}$. In this case, it follows by Kakutani's theorem that any bounded sequence $\{x_n\}$ has a subsequence which is w-convergent. 

Finding critical points of functions in infinite Banach spaces has many applications, among them is finding solutions to PDE. For example, a fundamental technique in PDE is to use a weak formulation. For example, $u$ is a $C^2$ solution to Poisson's equation $-\nabla ^2u=g$ on a domain $\Omega \subset \mathbb{R}^d$ with $u|_{\partial \Omega }=0$, if and only if for all smooth functions $v$ with the same boundary condition we have
\begin{eqnarray*}
-\int _{\Omega}(\nabla ^2u)vdx = \int _{\Omega}gv.
\end{eqnarray*}
 Now, by using integration by parts, one reduce the above to 
 \begin{eqnarray*}
 \int _{\Omega}\nabla u\nabla vdx-\int _{\Omega}gvdx=0.
 \end{eqnarray*}
 One can then construct a Sobolev's space of measurable functions $H^{1}_0(\Omega)$ of functions with weak zero boundary conditions, and define a norm
 \begin{eqnarray*}
 ||u||_*^2=\int _{\Omega}||\nabla u(x)||^2dx.
 \end{eqnarray*}
It turns out that $H^{1}_0(\Omega)$ is a Banach space, and is indeed a Hilbert space, and hence is reflexive. To find a solution to the original Poisson's equation, one proceeds in  two steps. In Step 1, one finds a weak solution $u$ in $H^{1}_0(\Omega)$, which is in general only a measurable function. In Step 2, one uses regularity theory to show that solutions from Step 1 are indeed in $C^2$ and hence are usual solutions. Now, we can relate Step 1 to finding critical points in the following way. We define a function $f:H^{1}_0(\Omega )\rightarrow \mathbb{R}$ by the formula
\begin{eqnarray*}
f(u)=\frac{1}{2}\int _{\Omega}||\nabla u(x)||^2dx-\int _{\Omega}gu.
\end{eqnarray*}
 We can check that $f$ is a $C^1$ function, and that $u$ is a weak solution of Poisson's equation iff $u$ is a critical point of $f$. 
 
Because of their importance in many aspects of science and life, there have been a lot of research on finding solutions of PDE. We concentrate here the approach of using numerical methods to solve PDE. There is a comprehensive list of such methods in \cite{numPDE}, such as Finite difference method, Method of lines, Finite element method, Gradient discretization method and so on. It seems, however, that numerical methods for finding critical points, which are powerful tools in finite dimensional spaces, are so far not much used. One reason may well be that numerical methods for finding critical points themselves are not much developed in infinite dimension. Indeed, it seems that most of the (convergence) results in this direction so far is based on Banach's fixed point theorem for contracting mappings, and applies only to maps which have some types of monotonicity. For systematic development in GD, one can see for example the very recent work \cite{geiersbach-scarinci}, where the problem is of stochastic nature, the space is a Hilbert space, and $f$ is either in $C^{1,1}_L$ (in which case only the results $\lim _{n\rightarrow\infty}||x_{n+1}-x_n||=0$ is proven); or when with additional assumption on learning rates $\sum _n\delta _n=\infty$ and $\sum _{n}\delta _n^2<\infty$, where convergence to $0$ of $\{\nabla f(x_n)\}$ can be proven. Some other references are \cite{gallego-etal} (where the function is assumed to be strongly convex) and \cite{blank-rupprecht} (where the VMPT method is used), where again only $\lim _{n\rightarrow\infty}||\nabla f(z_n)||=0$ is considered, and no discussion of (weak or strong) convergence of $\{x_n\}$ itself or avoidance of saddle points are given.   

In this paper, we will develop more the Backtracking GD method in infinite dimensional Banach spaces. To motivate the assumptions which will be imposed later on the functions $f$, we present here some further pathologies one faces in infinite dimensions. The first pathology is that for a continuous function $f:X\rightarrow \mathbb{R}$, it may happen that $f$ is unbounded on bounded subsets of $X$. The second pathology is that it may happen that $\{x_n\}$ w-converges to $x$, but $\nabla f(x_n)$ does not w-converge to $\nabla f(x)$. One needs to avoid this pathology, at least in the case when additionally $\lim _{n\rightarrow\infty}||\nabla f(x_n)||=0$, since in general the sequence $\{x_n\}$ one constructs from a GD method is not guaranteed to (strongly) converge, and hence at most can be hoped to w-converge to some $x$. If it turns out  that $\nabla f(x)\not= 0$, one cannot use GD method to find critical points. The second pathology is prevented by the following condition: 

{\bf Condition C.} Let $X$ be a Banach space. A $C^1$ function $f:X\rightarrow \mathbb{R}$ satisfies Condition C if whenever $\{x_n\}$ weakly converges to $x$ and $\lim _{n\rightarrow\infty}||\nabla f(x_n)||=0$, then $\nabla f(x)=0$.   

From a purely optimisation interest, Condition C is satisfied by quadratic functions, such as those considered on page 63 in \cite{nesterov} where Nesterov constructs some examples to illustrate the pathology with using GD in infinite dimensions. Indeed, if $f\in C^1$ is quadratic, then there is $A:X\rightarrow X^*$ a bounded linear operator and $b\in X^*$, so that $\nabla f(x)=Ax+b$. It is known that $A$ is then continuous in the weak topology. Therefore, if $\{x_n\}$ w-converges to $x$, then $\{Ax_n+b\}$ w-converges to $Ax+b$. In particular, $0=\liminf _{n\rightarrow\infty}||\nabla f(x_n)||\geq ||\nabla f(x)||$. Thus we have $\nabla f(x)=0$, and hence Condition C is satisfied.  

More interestingly, Condition C is valid for functions $f\in C^1$ which are convex. Indeed, assume that $f$ is in $C^1$ and is convex, and that $\{x_n\}$ w-converges to $x$ so that $\lim _{n\rightarrow\infty}||\nabla f(x_n)||=0$. Then since $f$ is continuous, it follows that $f$ is weakly lower semicontinuous (see Proposition 7b in \cite{browder}), that is $\liminf _{n\rightarrow\infty}f(x_n)\geq f(x)$. Now, since $f$ is convex and $\lim _{n\rightarrow\infty}||\nabla f(x_n)||=0$, for every $y\in X$ we have
\begin{eqnarray*}
f(y)\geq \limsup _{n\rightarrow\infty} [f(x_n)+<\nabla f(x_n),x_n-x>]=\limsup _{n\rightarrow\infty} f(x_n)\geq f(x).
\end{eqnarray*}
Therefore, $x$ is a  (global) minimum of $f$ and hence $\nabla f(x)=0$. 

It can be checked that the function associated to Poisson's equation satisfies Condition C (for example, because it is quadratic). More generally, the following class of non-linear PDE, see Definition 2 in \cite{browder}, satisfies Condition C:

{\bf Class $(S)_+$.} A $C^1$ function $f:X\rightarrow \mathbb{R}$ is of class $(S)_+$ if whenever $\{x_n\}$ w-converges to $x$ and $\limsup _{n\rightarrow\infty}<\nabla f(x_n),x_n-x>\leq 0$, then $\{x_n\}$ (strongly) converges to $x$. 

Indeed, assume that $f$ is in Class $(S)_+$, and $\{x_n\}$ w-converges to $x$ so that $\lim _{n\rightarrow\infty}$ $||\nabla f(x_n)||$ $=0$. Then, since $<\nabla f(x_n),x_n-x>\leq ||\nabla f(x_n)||\times ||x_n-x||$ and $\{||x_n-x||\}$ is bounded,  we have $\limsup _{n\rightarrow\infty}<\nabla f(x_n),x_n-x>\leq 0$. Therefore, by definition of Class $(S)_+$, we have $\{x_n\}$ (strongly) converges to $x$. Then, $\nabla f(x)=\lim _{n\rightarrow\infty}\nabla f(x_n)=0$. Hence Condition C is satisfied. ({\bf Remark.} Here we obtain the same conclusion under the weaker condition that $\nabla f$ is only demi-continuous. In this case, by definition, since $x_n$ strongly converges to $x$, we have that $\nabla f(x_n)$ weakly converges to $\nabla f(x)$. It follows that $||\nabla f(x)||\leq \liminf _{n\rightarrow\infty}||\nabla f(x_n)||$. Hence, if we have $\lim _{n\rightarrow\infty}||\nabla f(x_n)||=0$, it follows also that $\nabla f(x)=0$.)     

The interest in Class $(S)_+$ is that one can define topological degree for functions in this class, and hence can use to obtain lower bounds on the number of critical points. Moreover, Class $(S)_+$ includes many common PDE, such as Leray-Schauder where $Id-\nabla f$ is compact. 

Now we state one concise version of our main results. We assume that there is given a canonical isomorphism between $X$ and its dual $X^*$, for example when $X$ is a Hilbert space (see Proposition 8 in \cite{browder}, discussed in comments below, for how to deal with in the case $X$ is not Hilbert).  In the comments afterwards, we will discuss about extensions which are more complicated to state.

\begin{theorem} Let $X$ be a reflexive Banach space and $f:X\rightarrow \mathbb{R}$ be a $C^2$ function which satisfies Condition C. Moreover, we assume that for every bounded set $S\subset X$, then $\sup _{x\in S}||\nabla ^2f(x)||<\infty$. We choose a point $x_0\in X$ and construct by the Local Backtracking GD procedure (in the infinite dimensional setting) the sequence $x_{n+1}=x_n-\delta (x_n)\nabla f(x_n)$. Then we have: 

1) Every cluster point of $\{x_n\}$, in the {\bf weak} topology, is a critical point of $f$. 

2) Either $\lim _{n\rightarrow\infty}f(x_n)=-\infty$ or $\lim _{n\rightarrow\infty}||x_{n+1}-x_n||=0$. 

3) Here we work with the weak topology. Let $\mathcal{C}$ be the set of critical points of $f$. Assume that $\mathcal{C}$ has a bounded component $A$. Let $\mathcal{B}$ be the set of cluster points of $\{x_n\}$. If $\mathcal{B}\cap A\not= \emptyset$, then $\mathcal{B}\subset A$ and $\mathcal{B}$ is connected.     

4) Assume that $f$ has at most countably many saddle points. Then for generic choices of $\alpha ,\beta ,\delta _0$ and the initial point $x_0$, if the sequence $\{x_n\}$ converges - in the {\bf weak} topology, then the limit point cannot be a saddle point. 

5) Assume that $X$ is separable. Then for generic choices of $\alpha ,\beta ,\delta _0$ and the initial point $x_0$, if the sequence $\{x_n\}$ converges - in the {\bf weak} topology, then the limit point cannot be a saddle point. 
\label{TheoremMain}\end{theorem}
 \begin{proof} By the assumption on the Hessian of $f$, we see that $f$ satisfies the conditions needed to apply the Local Backtracking GD procedure, where one define $r(x)=1+||x||$ and $L(x)=\sup _{z\in B(x,r(x))}||\nabla ^2f(z)||$. Let $\{x_n\}$ be a sequence constructed by  the Local Backtracking GD procedure.
 
 1) By the usual estimates in finite dimensional Backtracking GD,  that whenever $K\subset X$ is a bounded set then $\inf _{x\in K}\hat{\delta}(x)>0$, we have that if $\{x_{n_j}\}$ w-converges to $x$, then $\lim _{j\rightarrow\infty}||\nabla f(x_{n_j})||=0$. Then by Condition C, we have that $\nabla f(x)=0$, as wanted. 
 
 2) This follows as in \cite{truong-nguyen}, by using that $\hat{\delta}(x_n)$ is bounded from above. (One can also use ideas in \cite{truong2} where $\hat{\delta}(x_n)$ is allowed to be unbounded.)
 
3) By  Condition C, it follows that $\mathcal{B}\subset \mathcal{C}$ and that $\mathcal{C}$ is w-closed.  We let $X_0\subset X$ be the closure of the vector space generated by $\{x_n\}$. Then $X_0$ is also reflexive, and moreover, it is separable. Moreover, $\mathcal{B}\subset X_0$, and hence $\mathcal{B}\subset \mathcal{C}\cap X_0$. Therefore, from now on, we can substitute $X$ by $X_0$, and substitute $\mathcal{C}$ by $\mathcal{C}\cap X_0$.  We define $\mathcal{C}'=\mathcal{C}\backslash A$. 

We define $\mathcal{B}_0=\mathcal{B}\cap A$, which is a non-empty w-closed set. We will show first that if $z\in \mathcal{B}$ then $z\in \mathcal{B}_0$. We note that $A,\mathcal{B}_0$ are compact in w-topology, and $\mathcal{C}'$ is closed in w-topology, and $A\cap \mathcal{C}'=\emptyset$. Recall that for each $z\in X$, the w-topology has a basis for open neighbourhoods of $z$ the sets of the form $\{x:~||h(x)-h(z)||<\epsilon\}$, where $\epsilon >0$ is a constant and $h:X\rightarrow \mathbb{R}^m$ is a bounded linear map to a finite dimensional space $\mathbb{R}^m$. Therefore, we can find an $\epsilon >0$ and a bounded linear map $h:X\rightarrow \mathbb{R}^m$ so that: for all $x'\in \mathcal{C}'$ and $x\in A$, then $|h(x)-h(z)|>\epsilon$. In particular, $h(A)\cap h(\mathcal{C}')=\emptyset$. Hence, $h(A)$ is a bounded component of $h(\mathcal{C})$. 

By 2) we have that $\lim _{n\rightarrow\infty}||x_{n+1}-x_n||=0$, and hence the same is true for the sequence $\{h(x_n)\}\subset \mathbb{R}^m$. Therefore, by the arguments in \cite{truong-nguyen}, by using the real projective space $\mathbb{P}^m$, we have that the closure in $\mathbb{P}^m$ of the set of  cluster points of $\{h(x_n)\}$  is connected. Note also that the cluster points of $\{h(x_n)\}$ is exactly $h(\mathcal{B})$, is contained in $h(\mathcal{C})$ and has a non-empty intersection with $h(A)$. It follows that $h(\mathcal{B})$ must be contained in $h(A)$ and is connected. In particular, $\mathcal{B}\cap \mathcal{C}'=\emptyset$, as claimed. 

Now to finish the proof, we will show that $\mathcal{B}$ is itself connected.  This can in fact be done as above, by using one connected component of $\mathcal{B}$ in the place of $\mathcal{B}$. 

4) For each generalised saddle point $z$ of $f$, we choose a non-zero $v(z)$ which is an eigenvector of $\nabla ^2f(z)$ with a negative eigenvalue. Construct the $2$-dim vector subspace $W_z=\mathbb{R}z\oplus \mathbb{R}v(z)$ of $X$, and let $pr_z:~X\rightarrow W_z$ be a bounded linear projection which is available by Hahn-Banach theorem. 

We will assume that $\alpha,\beta $ and $\delta _0$ are chosen for which $\alpha /L(z)$ does not belong to the discrete set $\{\beta ^n\delta _0:~n=0,1,2,\ldots \}$ for every generalised saddle point of $\mathcal{B}$. Then we look at the dynamical system on $W_z$:
\begin{eqnarray*}
H_z(x)=x-\hat{\delta}(x)pr_z(\nabla f(x)):~W_z\rightarrow W_z.
\end{eqnarray*}
As in \cite{truong}, the above assumptions on $\alpha ,\beta$   and $\delta$ implies that $H_z$ is a local diffeomorphism near $z\in W_z$, and hence we have by Central-Stable manifold theorem (see Theorem III.7 in \cite{shub}), there is a local Central-Stable manifold (of dimension $<2$) in a small neighbourhood $U_z\subset W_z$ of $z$, such that if $x\in U_z$ and $H_z^n(x)\in U_z$ for all $n$, then $x$ is contained in the local Central-Stable manifold. Also, as in \cite{truong}, locally $H_z$ is one of a finite number of Lipschitz continuous maps, and hence there is an exceptional set $\mathcal{E}_z\subset W_z$ with the following properties: i) $\mathcal{E}_z$ is closed, ii) $\mathcal{E}_z$ is nowhere dense, and iii) if $x\in W_z\backslash \mathcal{E}_z$, then the sequence $\{H_z^n(x)\}$ {\bf does not} converge to $z$. We can in fact show also that $\mathcal{E}_z$ has Lebesgue measure $0$, and in the remarks after the proof of this theorem we will discuss stronger properties for the exceptional set $pr_{z}^{-1}(\mathcal{E}_z)\subset X$ in view of the definitions in \cite{hunt-etal}.  

Now assume that the sequence $\{x_n\}$, constructed by the Local Backtracking GD procedure, w-converges to  a generalised saddle point $z$ of $f$. Then (as points in $X$, and not just in $W_z$), $pr_z(x_n)$ strongly converges to $z$, and hence (because $f\in C^1$) $\nabla f(pr_z(x_n))$ strongly converges to $\nabla f(z)=0$. Therefore, since $\hat{\delta}(x)$ is bounded from above by $\delta _0$, we have $H_{z}(pr_z(x_n))$ converges to $z$. This means that $pr_z(x_0)\in \mathcal{E}_z$ and hence $x_0\in pr_z^{-1}(\mathcal{E}_z)$. Since $pr_z$ is a bounded linear projection, we have that $ pr_z^{-1}(\mathcal{E}_z)$  is also closed and nowhere dense. Thus we see that if $x_0$ is chosen outside the union of these $pr_z^{-1}(\mathcal{E}_z)$, where $z$ runs  all over generalised saddle points of $f$ - this union is also nowhere dense since we assumed that there are only at most countably many such $z$'s, then the sequence $\{x_n\}$ cannot w-converge to any generalised saddle point.  

5) For each generalised saddle point of $z$, we construct as in part 4 the space $W_z$ and the map $H_z$, together with the neighbourhood $U_z\subset W_z$ of $z$ where a local Central-Stable manifold exists, as well as a bounded linear projection $pr_z:~X\rightarrow H_z$. 

Since $X$ is reflexive and separable, it follows that $X^*$ is also separable. Then (see point iii) in \cite{dance-sims}) it follows that any bounded subset of $X$, in the weak topology, is metrizable. Also, since $X$ is separable, it follows that $X$ is hereditarily Lindel\"of. That is, any open cover of a subspace of $X$, in the weak topology, has a countable open subcover. Then we can find countably many generalised saddle points $\{z_i\}_{i=1,2,\ldots }$ such that $\{pr_{z_i}^{-1}(U_{z_i})\}_{i=1,2,\ldots }$ is an open subcover of the set of all generalised saddle points for the open cover $\{pr_z^{-1}(U_z)\}$ where z runs all over generalised saddle points. Then we can proceed as in part 4) with the points $\{z_i\}_{i=1,2,\ldots }$.  
\end{proof}

{\bf Remarks:}

1) For general Banach spaces,  we cannot regard $\nabla f(x_n))\in X^*$ as an element of $X$, and hence the update rule in the theorem is not legit. However,  in parts 1, 2 and 3 of the theorem, we can choose $x_{n+1}=x_n-\hat{\delta}(x_n)v(x_n)$ provided $v(x_n)\in X$ is chosen such that Armijo's condition is satisfied. To this end, we choose (by Hahn-Banach theorem), under the assumption that $X$ is reflexive, a point $v(x_n)\in X$ so that $<\nabla f(x_n),v(x_n)>=||\nabla f(x_n)||^2$ and $||v(x_n)||=||\nabla f(x_n)||$. Hence we can also work with Frechet differentiation as in \cite{blank-rupprecht}. Likewise, we can use other versions of Backtracking GD, under less restrictions on the function $f$.

Note that as of current, for parts 4 and 5 of the theorem we need to use Stable-Central manifold, and hence if the map $x\mapsto x-\hat{\delta}(x)v(x)$ is not $C^1$ near saddle points, then it is difficult to proceed, as of current. See however point 5 below for an idea on using duality mappings to deal with this in general. 

2) This theorem can be extended to several other versions of Backtracking GD, including Unbounded Backtracking GD. The conclusion of parts 4 and 5 of the theorem are also valid if we use the strong topology instead of weak topology. This is because we can again use Lindel\"of lemma and the fact that Stable-Central manifold theorem is also valid in the Banach setting (see Theorem III.8 in \cite{shub}). 

3) We note that while there is no analog of Lebesgue measures in infinite dimensions, there are good similar notions of "almost every" and "zero measure", which can be used to quantify a statement about a property which  is "generically" true. We recall here the notion of "shyness" in Banach spaces in \cite{hunt-etal}, and will show that the exceptional set for $x_0$ in parts 4 and 5 of Theorem \ref{TheoremMain} are "shy". 

Let $X$ be a Banach space. Then a Borel set $S\subset X$ is called "shy" if there exists a measure $\mu$ on $X$ with the following properties: i) There exists a compact set $U\subset X$ such that $0<\mu (U)<\infty$; ii) For every $x\in X$, then $\mu (x+S)=0$. Note that Fact 2' in \cite{hunt-etal} says that the complement of a "shy" set is dense. Note also that by Fact 6 in \cite{hunt-etal}, if $X$ has finite dimension, then $S\subset X$ is "shy" iff $S$ has Lebesgue measure $0$.  

Now we show that the sets $pr_z^{-1}(\mathcal{E}_z)$ appearing in the proofs of parts 4 and 5 of Theorem \ref{TheoremMain} are "shy". In fact, since $W_z$ is finite dimensional and $E_z\subset W_z$ has Lebesgue measure $0$. We choose $\mu _z$ to be the Lebesgue measure on $W_z$, and will show that for all $x\in X$, then $\mu _z(x+pr_z^{-1}(\mathcal{E}_z))=0$. To this end, since $\mu _z$ has support in $W_z$, it follows from the fact that $pr_z:X\rightarrow W_z$ is a bounded linear projection that 
\begin{eqnarray*}
\mu _z(x+pr_z^{-1}(\mathcal{E}_z))=\mu _z([x+pr_z^{-1}(\mathcal{E}_z)]\cap W_z)= \mu _z(pr_z(x)+\mathcal{E}_z)=\mu _z(\mathcal{E}_z)=0.
\end{eqnarray*}

By Fact 3" in \cite{hunt-etal}, we have that a countable union of $pr_z^{-1}(\mathcal{E}_z)$ is also "shy".  

4) In part 3) of Theorem \ref{TheoremMain}, if $f$ has at most countably many critical points (including Morse functions), then the conclusion is simple to state: Either $\lim _{n\rightarrow\infty}||x_n||=\infty$ or $\{x_n\}$ w-converges to a critical point of $f$. 

5) Duality mapping: 

If we are willing to redefine the norms on $X$, then $v(x_n)$ (in point 1 above) is unique and depends continuously on $x_n$. The argument presented here follows Proposition 8 in \cite{browder}. Assume that $X$ is reflexive. Then we can renorm both $X$ and $X^*$ so that both of them are locally uniformly convex. Then, we have a well-defined map $J:X^*\rightarrow X$ with the following properties: $<y,J(y)>=||y||^2$ and $||J(y)||=||y||$ for all $y\in X^*$. This map is called normalized duality mapping, and it is bicontinuous. Moreover, it is monotone, that is $<y_1-y_2,J(y_1)-J(y_1)>\geq 0$, for all $y_1,y_2\in X^*$. Indeed, 
\begin{eqnarray*}
<y_1-y_2,J(y_1)-J(y_1)>&=& [(||y_1||-||y_2||)^2] + [||y_1||\times ||J(y_2)||-<y_1,J(y_2)>] \\
&+& [||y_2||\times ||J(y_1)||-<y_2,J(y_1)>],
\end{eqnarray*}
and each square bracket is non-negative. Hence, the map $J$ is in Class $(S)_+$. 

Our GD update rule is now: $x\mapsto x-\delta (x)J(\nabla f(x))$, which is continuous. Even though it is not $C^1$ in general, it seems good enough that conclusions of  parts 4 and 5 of Theorem \ref{TheoremMain} should follow. For example, if $X^*$ is the space $l^p$ ($1<p<\infty$), i.e. the space of sequences $x=(x^{(1)},x^{(2)},\ldots )$ so that $\sum _{j}|x^{(j)}|^p<\infty$, with norm $||x||=(\sum _j|x^{(j)}|^p)^{1/p}$, then the normalized duality mapping is: 
\begin{eqnarray*}
J(x)=\frac{1}{||x||^{p-2}}(|x^{(1)}|^{p-1}sign (x^{(1)}),|x^{(2)}|^{p-1}sign (x^{(2)}),\ldots )
\end{eqnarray*}
Assume that $p\geq 2$. The map $J(x)$ is $C^1$ outside of the hyperplanes $\{x^{(j)}=0\}$, and its derivative is locally bounded near each points of these hyperplanes. (To see this, the readers can simply work out the differentiability of the following map in $2$ variables $|s|^{p-2}s/(|s|^p+|t|^p)^{(p-2)/p}$.)

\end{document}